\documentclass{amsart}
\usepackage{amsmath}
\usepackage{paralist}
\usepackage{graphics} %% add this and next lines if pictures should be in esp format
\usepackage{epsfig} %For pictures: screened artwork should be set up with an 85 or 100 line screen
\usepackage[colorlinks=true]{hyperref} %% Warning: when you first run your tex file, some errors might occur, please just

\newtheorem{theorem}{Theorem}%[section]

\newtheorem*{lemma}{Lemma}

\theoremstyle{definition}

\def\be{\begin{equation}}        
\def\ee{\end{equation}}
\numberwithin{equation}{section}

\title[The Dirac Delta Function]
      {Integral and Series Representations \\of the Dirac Delta Function}

\author[Y. T. Li and R. Wong]{}


 \email{50008430@student.cityu.edu.hk}
 \email{mawong@cityu.edu.hk}

\begin{document}
\maketitle

%% Enter the first author's name and address:
\centerline{\scshape Y. T. Li }
\medskip
{\footnotesize
 %% please put the address of the first author
 \centerline{ Department of Mathematics, City University of
Hong Kong,}
   \centerline{Tat Chee Avenue, Kowloon, Hong Kong.}
} %% Do not forget to end the {\footnotesize by the sign }

\medskip

\centerline{\scshape R. Wong}
\medskip
{\footnotesize
 %% please put the address of the second author
 \centerline{ Liu Bie Ju Centre for Mathematical Sciences, City University of
Hong Kong,  }
   \centerline{Tat Chee Avenue,  Kowloon, Hong Kong.}
} %

\bigskip

\begin{abstract}
Mathematical justifications are given for several integral and
series representations of the Dirac delta function which appear in
the physics literature. These include integrals of products of
Airy functions, and of Coulomb wave functions; they also include
series of products of Laguerre polynomials and of spherical
harmonics. The methods used are essentially based on the
asymptotic behavior of these special functions.
\end{abstract}

\section{Introduction}

 The Dirac delta function $\delta(x)$ has been used in physics well before the theory of distributions (generalized functions) was introduced by mathematicians. The manner in which physicists  used this function was to define it by the equations
\begin{equation}
\delta(x-a)=0, \qquad\qquad x\neq a,
\end{equation}
and
\begin{equation}
\int^\infty_{-\infty}\phi(x)\delta(x-a)dx=\phi(a), \qquad\qquad
a\in \mathbb{R},
\end{equation}
for any continuous function $\phi(x)$ on  $\mathbb{R}$. However,
mathematically, these two equations are inconsistent in the
classical sense of a function and an integral, since the value of
the integral of a function which  is zero everywhere except for a
finite number of points should be zero. There are now two
mathematically meaningful approaches to help us interpret the
delta function given in $(1.1)-(1.2)$. One approach is to consider
$\delta_a:=\delta(x-a)$ as a continuous linear functional acting
on a space of smooth functions with rapid decay at $\pm \infty$,
and the action of $\delta_a$ on a particular function $\phi(x)$ is
given the value $\phi(a)$; see \cite[p.141]{Rudin} and
\cite[p.77]{Schwartz}. The other approach is to find a sequence of
functions $\delta_n(x-a)$ such that
\begin{equation}
\lim_{n\to\infty}\int^\infty_{-\infty}\delta_n(x-a)\phi(x)dx=\phi(a),
\qquad\qquad  a\in\mathbb{R};
\end{equation}
see \cite[p.55]{Jones} and \cite[p.17]{Lighthill}. Such a sequence
is called a {\it delta sequence} and we write, symbolically,
 \begin{equation}
\lim_{n\to\infty}\delta_n(x-a)=\delta(x-a),   \qquad\qquad
x\in\mathbb{R}.
\end{equation}
It seems that  the second approach  is more acceptable to
physicists and applied mathematicians.

Recently, in the process of preparing some material for the major
project ``NIST Handbook of Mathematical Functions \cite{OLCB}'',
we encountered some very interesting integral and series
representations of the delta function which need mathematical
justification. For instance, in \cite[p.696]{AW} the formula
 \begin{equation}
\int^\infty_0 xtJ_\nu(xt)J_\nu(at)dt=\delta(x-a),   \qquad\qquad
\textup{Re}\; \nu>-1,\; \;x>0,\;\; a>0,
\end{equation}
appears, where $J_\nu(x)$ is the Bessel function of the first
kind, and in \cite[Eq.(122)]{Seaton} one finds the integral
representation
 \begin{equation}
\int^\infty_0 s(x, l; t)s(a, l; t)dt=\delta(x-a),  \qquad\qquad
a>0,\;\; x>0,
\end{equation}
where $s(x, l; t)$ is the Coulomb wave function. A recent
reference  \cite[p.57]{VS} on the Airy function $\textup{Ai}(x)$
also gives the formula
\begin{equation}
\int^\infty_{-\infty}
\textup{Ai}(t-x)\textup{Ai}(t-a)dt=\delta(x-a).
\end{equation}
While physicists may find these representations convenient to use
in applications, mathematicians would, in general, feel uneasy or
even disturbed to see these formulas being used since the
integrals in $(1.5)-(1.7)$ are all divergent. Thus, it would seem
meaningful and necessary to give a mathematical justification for
these representations, and this is exactly the purpose of the
present paper.

There are also some series representations for the delta function.
These include the following:
 \begin{equation}
\sum^\infty_{k=0}\biggl(k+\frac12\biggr)P_k(x)P_k(a)=\delta(x-a),
\end{equation}
 \begin{equation}
\frac{e^{-(x^2+a^2)/2}}{\sqrt{\pi}}\sum^\infty_{k=0}\frac{1}{2^k
k!}H_k(x)H_k(a)=\delta(x-a),
\end{equation}
and
 \begin{equation}
e^{-(x+a)/2}\sum^\infty_{k=0}L_k(x)L_k(a)=\delta(x-a),
\end{equation}
where $P_k(x)$, $H_k(x)$ and $L_k(x)$ are, respectively, the
Legendre, Hermite and Laguerre polynomials. Equations
$(1.8)-(1.10)$ are special cases of an equation  in Morse and
Feshbach  \cite[p.729]{MF}. Another series representation is given
in \cite[p.792]{AW}; that is,
 \begin{equation}
 \begin{split}
\sum^\infty_{k=0}\sum^k_{l=-k}Y_{kl}(\theta_1,
\phi_1)Y^*_{kl}(\theta_2, \phi_2)
&=\frac{1}{\sin\theta_1}\delta(\theta_1-\theta_2)\delta
(\phi_1-\phi_2)   \\
  &=\delta(\cos\theta_1-\cos\theta_2)\delta
(\phi_1-\phi_2),
\end{split}
\end{equation}
where the functions $Y_{kl}(\theta, \phi)$ are the spherical
harmonics (see \cite[p.788]{AW}) and the asterisk ``$\ast$''
denotes complex conjugate.

The orthogonal polynomials in $(1.8)-(1.10)$ and the orthogonal
function in (1.11) are the eigenfunctions corresponding to the
eigenvalues (discrete spectrum) of some differential operators.
Likewise, the special functions in $(1.5)-(1.7)$ can be regarded
as the eigenfunctions associated with the continuous spectrum of
corresponding differential  operators. The proofs of the
representations in $(1.8)-(1.11)$ turn out to be much simpler than
the proofs of those in $(1.5)-(1.7)$. Indeed, we shall show that
the results in $(1.8)-(1.11)$ all follow from expansion theorems
in  orthogonal polynomials, whereas for the representations in
$(1.5)-(1.7)$ we need to provide some new arguments.

\section{A generalized Riemann-Lebesgue Lemma}
There are already several delta sequences in the literature. For
instance, we have
 \begin{eqnarray}
&\delta_n(x-a)=\sqrt{\dfrac{n}{\pi}}e^{-n(x-a)^2},
\\
&\delta_n(x-a)=\dfrac{n}{\pi}\dfrac{1}{1+n^2(x-a)^2}, \end{eqnarray}
and
 \begin{equation}
\delta_n(x-a)=\frac{1}{\pi}\frac{\sin n(x-a)}{x-a};
\end{equation}
see \cite[pp. 35-38]{GS} and \cite[pp. 5-13]{Kanwal}. To verify
whether a given sequence of functions is a delta sequence, one can
apply the criteria given in \cite[p.34]{GS}. If the function
$\phi(x)$ in (1.3) is only piecewise continuous in $\mathbb{R}$,
then this equation becomes
\begin{equation}
\lim_{n\to\infty}\int^\infty_{-\infty}\delta_n(x-a)\phi(x)dx=\frac{1}{2}[\phi(a^+)+\phi(a^-)],
\qquad \qquad a\in \mathbb{R};
\end{equation}
see \cite[p.16]{Kanwal}.

For convenience in our later argument, we also state and prove the
following result.
\begin{lemma}
\textup{(A generalized Riemann-Lebesgue lemma)}. Let $g(x, R)$ be
a continuous function of $x\in(A,  B)$ and uniformly bounded for
$R>0$. If
 \begin{equation}
 \lim_{R\to\infty} \int^{B'}_{A'}g(x, R)dx=0
\end{equation}
 for any $A'$ and $B'$ with $A<A'<B'<B$, then
 \begin{equation}
 \lim_{R\to\infty} \int^{B}_{A}\psi(x)g(x, R)dx=0
\end{equation}
for any integrable function $\psi(x)$ on the finite interval $(A,
B)$. If $A=0$ and $B=+\infty$, or if $A=-\infty$ and  $B=+\infty$,
then \textup{(2.6)} holds for any absolutely integrable function
$\psi(x)$ on the infinite interval $(A, B)$.
\end{lemma}

\begin{proof} First, from (2.5) it is easy to see that (2.6) holds
for step functions. Now, let $\psi(x)$ be an integrable function
on $(A, B)$. For any $\varepsilon>0$, we can always find a step
function $s(x)$ such that
 \begin{equation}
 \int^{B}_{A}|\psi(x)-s(x)|\,dx<\frac{\varepsilon}{2K},
\end{equation}
where $K=\max\{|g(x, R)|:x\in\mathbb{R}\;\textup{and}\;R>0\}$.
Choose $R_0>0$ so that
 \begin{equation}
\biggl|\int^{B}_{A}s(x)g(x, R)dx\biggr|<\frac{\varepsilon}{2}
\qquad \qquad \textup{for all}\;R\geq R_0.
\end{equation}
Write
 \begin{equation*}
\int^{B}_{A}\psi(x)g(x, R)dx=\int^{B}_{A}[\psi(x)-s(x)]g(x,
R)dx+\int^{B}_{A}s(x)g(x, R)dx.
\end{equation*}

%\newpage
\noindent On account of (2.7) and (2.8), we have
 \begin{equation*}
\biggl|\int^{B}_{A}\phi(x)g(x, R)dx\biggr|<\varepsilon
\end{equation*}
for all $R\geq R_0$. Since $\varepsilon$ is arbitrary, this proves
(2.6) when $A$ and $B$ are finite.

If the interval of integration is infinite, and if $\psi(x)$  is
absolutely integrable there, then we can choose finite numbers $A$
and $B$ such that the integral $\int\psi(x)g(x, R)dx$ outside the
interval  $(A, B)$ is small since $g(x, R)$ is uniformly bounded.
On  the finite interval  $(A, B)$, we can apply the result just
established.
\end{proof}

\section{Bessel function}
The Bessel function $J_\nu(xt)$ is a solution of the differential
equation
 \begin{equation}
\frac{d}{dt}\biggl(t\frac{dy}{dt}\biggr)+\biggl(x^2t-\frac{\nu^2}{t}\biggr)y=0.
\end{equation}
With $x$ replaced by $a$, one obtains a corresponding equation for
$J_\nu(at)$. Multiplying   equation (3.1) by $J_\nu(at)$ and the
corresponding equation for $J_\nu(at)$ by $J_\nu(xt)$, and
subtracting the two resulting equations, leads to
 \begin{equation}
 \begin{split}
(x^2-a^2)tJ_\nu(at) J_\nu(xt)=\frac{d}{dt}\{&at J_\nu(xt)J'_\nu(at)\\
&\hspace{-0.7mm}-xtJ_\nu(at)J'_\nu(xt)\}.
\end{split}
\end{equation}
Put
 \begin{equation}
\delta_R(x, a)=x\int^R_0 J_\nu(xt)J_\nu(at)t\,dt.
\end{equation}
An integration of (3.2) gives
\begin{equation}
\delta_R(x, a)=\frac{x}{x^2-a^2}\biggl[at\,
J_\nu(xt)J'_\nu(at)-xt\, J_\nu(at)J'_\nu(xt)\biggr]^R_0.
\end{equation}
From the ascending power series representation
\begin{equation}
 J_\nu(t)=\biggl(\frac{t}{2}\biggr)^{\hspace{-1mm}\nu}\sum^\infty_{n=0}\frac{(-1)^n}{\Gamma(n+\nu+1)n!}\biggl(\frac{t}{2}\biggr)^{\hspace{-1mm}{2n}},
\end{equation}
it can be shown that the leading terms in the series expansions of
$at\,J_\nu(xt)J'_\nu(at)$ and $xt\, J_\nu(at)J'_\nu(xt)$ cancel
out. Thus, the right-hand side of (3.4) vanishes at the lower
limit when $\nu>-1$, and we obtain
\begin{equation}
\delta_R(x, a)=\frac{x}{x^2-a^2}[aR J_\nu(xR)J'_\nu(aR)-xR
J_\nu(aR)J'_\nu(xR)].
\end{equation}
 \begin{theorem}
For $a>0$, $\nu>-1$ and any piecewise continuously differentiable
function $\phi(x)$ on $(0, \infty)$, we have
\begin{equation}
\lim_{R\to\infty}\int^{\infty}_{0}\phi(x)x \biggl(\int^{R}_{0}
J_\nu(xt)J_\nu(at)t \;dt\biggr)dx=\frac12[\phi(a^+)+\phi(a^-)],
\end{equation}
provided that

\textup{(i)} $\int^{\infty}_{1} x^{-\frac12}|\phi(x)|dx$ converges;\\
\indent \textup{(ii)} $\int^{1}_{0} x^{\frac12}|\phi(x)|dx$ converges when $\nu\geq -\frac12$, or \\
\indent \textup{(ii$'$)} $\int^{1}_{0} x^{\nu+1}|\phi(x)|dx$
converges when $-1<\nu <-\frac12$.
 \end{theorem}
\begin{proof} In view of the asymptotic formulas
\begin{equation}
J_\nu (xt)=\sqrt{\frac{2}{\pi xt}}\biggl[\cos
\biggl(xt-\frac{\nu\pi}{2}-\frac{\pi}{4}\biggr)+\varepsilon_1(x,
t)\biggr]
\end{equation}
and
\begin{equation}
J'_\nu (xt)=-\sqrt{\frac{2}{\pi xt}}\biggl[\sin
\biggl(xt-\frac{\nu\pi}{2}-\frac{\pi}{4}\biggr)+\varepsilon_2(x,
t)\biggr]
\end{equation}
where $\varepsilon_j(x, t)=O(1/t)$ as $t\to\infty$ uniformly for
$x\geq\delta>0$ and $j=1, 2$, there are constants $M_1>0$ and
$M_2>0$ such that
\begin{equation}
|J_\nu(aR)|\leq M_1 R^{-\frac12}, \qquad\qquad  |J'_\nu(aR)|\leq
M_1 R^{-\frac12}
\end{equation}
and
\begin{equation}
|J_\nu(xR)|\leq M_2 x^{-\frac12}R^{-\frac12}, \qquad\qquad
|J'_\nu(xR)|\leq M_2 x^{-\frac12}R^{-\frac12}
\end{equation}
for $x\geq1$ and $R\geq 1$. From (3.6), it follows that
\begin{equation*}
|\delta_R(x, a)|\leq
M_1M_2\frac{x^2}{|x^2-a^2|}\biggl(\frac{a}{x^{3/2}}+\frac{1}{x^{1/2}}\biggr).
\end{equation*}
Hence, for $b>\max\{a, 1\}$, we have
\begin{equation*}
\int^{\infty}_{b}|\delta_R(x, a)\phi(x)|dx\leq M_3
\int^{\infty}_{b}x^{-\frac12}|\phi(x)|dx,
\end{equation*}
where $M_3=M_1M_2(a+1)b^2/(b^2-a^2)$. Since the last integral is
convergent by condition (i), for any $\varepsilon>0$ there exists
a number $c>b$ such that
\begin{equation}
\int^{\infty}_{c}|\delta_R(x, a)\phi(x)|dx<\frac{\varepsilon}{2}.
\end{equation}

Let $0<\rho<\min\{a, 1\}$. To estimate the integral of
$\delta_R(x, a)\phi(x)$ on the interval $(0, \rho)$, we divide our
discussion into two cases: (i) $\nu\geq -\frac12$, and (ii)
$-1<\nu<-\frac12$. In the first case, we have from (3.5) a
positive constant $M'_4$ such that
\begin{subequations}
\begin{equation}
|J_\nu(xR)|\leq M'_4 (xR)^{\nu}\leq M'_4(xR)^{-\frac12}
\end{equation}
and
\begin{equation}
|J'_\nu(xR)|\leq M'_4 (xR)^{\nu-1}\leq
M'_4x^{-\frac32}R^{-\frac32}\leq  M'_4x^{-\frac32}R^{-\frac12}
\end{equation}
\end{subequations}
for $0<xR\leq 1$ and $R\geq 1$. From (3.8) and (3.9), we also have
\begin{equation}
|J_\nu(xR)|\leq  M''_4x^{-\frac12}R^{-\frac12},\qquad \qquad
|J'_\nu(xR)|\leq M''_4x^{-\frac12}R^{-\frac12}
\end{equation}
for $xR\geq 1$. Coupling (3.13) and (3.14), we obtain
\begin{equation*}
|J_\nu(xR)|\leq M_4 x^{-\frac12}R^{-\frac12},\qquad \qquad
|J'_\nu(xR)|\leq M_4x^{-\frac32}R^{-\frac12}
\end{equation*}
for $0<x\leq 1$ and $R\geq 1$, with $M_4=\max\{M'_4, M''_4\}$.
Thus,
\begin{equation*}
|\delta_R(x, a)|\leq M_1M_4(a+1)\frac{x^{\frac12}}{a^2-\rho^2},
\qquad \qquad 0<x\leq 1,
\end{equation*}
and
\begin{equation}
\int^{\rho}_{0}|\delta_R(x, a)\phi(x)|dx\leq M_5\int^\rho_0
x^{\frac12}|\phi(x)|dx,
\end{equation}
where $M_5=M_1M_4(a+1)/(a^2-\rho^2)$.

In the second case, there are constants $M'_6>0$ and $M''_6>0$
such that
\begin{equation*}
\begin{split}
&\hspace*{6mm}|J_\nu(xR)|\leq M'_6 (xR)^{\nu}\leq M'_6x^{\nu}R^{-\frac12}, \\
&|J'_\nu(xR)|\leq M'_6 (xR)^{\nu-1}\leq M'_6x^{\nu-1}R^{-\frac12}
\end{split}
\end{equation*}
for $0<xR \leq 1 $ and $R\geq 1$, and
\begin{equation*}
\begin{split}
&|J_\nu(xR)|\leq M''_6 x^{-\frac12}R^{-\frac12}\leq M''_6x^{\nu}R^{-\frac12}, \\
&|J'_\nu(xR)|\leq M''_6 x^{-\frac12}R^{-\frac12}\leq
M''_6x^{\nu-1}R^{-\frac12}
\end{split}
\end{equation*}
for $xR\geq 1$ and $R\geq 1$. With $M_6=\max\{M'_6, M''_6\}$, it
follows that
\begin{equation*}
|J_\nu(xR)|\leq M_6 x^{\nu}R^{-\frac12}\qquad\textup{and} \qquad
|J'_\nu(xR)|\leq M_6 x^{\nu-1}R^{-\frac12}.
\end{equation*}
Therefore,
\begin{equation*}
|\delta_R(x, a)|\leq M_1M_6\frac{1}{a^2-\rho^2}(a+1)x^{\nu+1}
\end{equation*}
and
\begin{equation}
\int^{\rho}_{0}|\delta_R(x, a)\phi(x)|dx\leq
M_7\int^{\rho}_{0}x^{\nu+1}|\phi(x)|dx,
\end{equation}
where $M_7=M_1M_6(a+1)/(a^2-\rho^2)$. On account of conditions
(ii) and (ii$'$), for any $\varepsilon>0$ there is a constant
$0<d<\rho$ such that
\begin{equation}
\int^{d}_{0}|\delta_R(x, a)\phi(x)|dx<\frac{\varepsilon}{2}.
\end{equation}

For $d<x<c$, a combination of (3.6), (3.8) and (3.9) gives
\begin{equation}
\begin{split}
\delta_R(x, a)=\frac{2}{\pi}\frac{x}{x^2-a^2}\biggl\{ &\sqrt{\frac{x}{a}}\sin \zeta(x, R)\cos\zeta(a, R) \\
&-\sqrt{\frac{a}{x}}\cos \zeta(x, R)\sin\zeta(a, R)+\varepsilon(x,
a; R)\biggr\},
 \end{split}
\end{equation}
where $c$ and $d$ are given in (3.12) and (3.17), respectively,
$\zeta(x, R)=xR-\frac{\nu \pi}{2}-\frac{ \pi}{4}$ and
$\varepsilon(x, a; R)=O(1/R)$ as $R\to \infty$ uniformly for $x\in
(d, c)$. Note that $\varepsilon(x, a; R)$ is continuously
differentiable, $\varepsilon(x, a; R)/(x-a)$ is continuous in $x$
and uniformly bounded in $x$ and $R$, and $\phi(x)$ is piecewise
continuous in $(d, c)$. Hence,
 \begin{equation}
\lim_{R\to\infty}\int^{c}_{d}\frac{\varepsilon(x, a;
R)}{x^2-a^2}x\phi(x)dx=0.
\end{equation}
By the Riemann-Lebesgue lemma, we also have
 \begin{equation}
\lim_{R\to\infty}\int^{c}_{d} \biggl( \sqrt{\frac{x}{a}}-1\biggr)
\frac{\sin\zeta(x, R)}{x^2-a^2}x\phi(x)dx=0
\end{equation}
and
 \begin{equation}
\lim_{R\to\infty}\int^{c}_{d} \biggl( \sqrt{\frac{a}{x}}-1\biggr)
\frac{\cos\zeta(x, R)}{x^2-a^2}x\phi(x)dx=0.
\end{equation}
A combination of the results in $(3.18) - (3.21)$ yields
 \begin{equation*}
\begin{split}
\lim_{R\to\infty}\int^{c}_{d} \delta_R(x, a)\phi(x)dx &= \lim_{R\to\infty}\frac{2}{\pi}  \int^{c}_{d} \frac{\sin\{\zeta(x, R)-\zeta(a, R)\}}{x^2-a^2}x\phi(x)dx\\
&= \lim_{R\to\infty}  \int^{c}_{d} \frac{\sin
R(x-a)}{\pi(x-a)}\cdot \frac{2x}{x+a}  \phi(x)dx
 \end{split}
\end{equation*}
On account of Jordan's theorem on the Dirichlet kernel
\cite[p.473]{Apostol}, we conclude
 \begin{equation}
\lim_{R\to\infty}\int^{c}_{d} \delta_R(x,
a)\phi(x)dx=\frac12[\phi(a^-)+\phi(a^+)].
\end{equation}
Since the number $\varepsilon$ in (3.12) and (3.17) is arbitrary,
it follows from (3.22) that
 \begin{equation*}
\lim_{R\to\infty}\int^{\infty}_{0} \delta_R(x,
a)\phi(x)dx=\frac12[\phi(a^-)+\phi(a^+)],
\end{equation*}
which is equivalent to (3.7).
\end{proof}

\section{Coulomb wave function}
The Coulomb wave function $s(x, l; r)$ is a solution of the
Coulomb wave equation
 \begin{equation}
\frac{d^2y}{dr^2}+\biggl\{x+\biggl( \frac{2}{r}-
\frac{l(l+1)}{r^2}\biggr)\biggr\}y=0,
\end{equation}
that satisfies the initial conditions
 \begin{equation}
s(x, l; 0)=s'(x, l; 0)=0
\end{equation}
and has the asymptotic behavior
 \begin{eqnarray}
&s(x, l; r)=\dfrac{1}{\sqrt{\pi}}x^{-\frac14}[\sin \zeta(x, l;
r)+\varepsilon_1(x, l; r)],
\\
&s'(x, l; r)=\dfrac{1}{\sqrt{\pi}}x^{\frac14}[\cos \zeta(x, l;
r)+\varepsilon_2(x, l; r)],
\end{eqnarray}
where $\varepsilon_j(x, l; r)=O(1/r)$ as $r\to \infty$ uniformly
for $x\geq\delta>0, j=1, 2$, and
 \begin{equation}
\zeta(x, l;
r)=kr+\frac{1}{k}\textup{ln}\;(2kr)-\frac{l\pi}{2}+\arg\Gamma\biggl(
l+1-\frac{i}{k}\biggr)
\end{equation}
with $k=\sqrt{x}$; see \cite[p.236]{Seaton}. For $a>0$ and $x>0$,
we define
 \begin{equation}
 \delta_R(x, a)=\int^R_0 s(x, l; r)s(a, l; r)dr.
\end{equation}
From (4.1) and (4.2), one can show as in Sec. 3 that
 \begin{equation}
 \delta_R(x, a)=\frac{ s(x, l; R)s'(a, l; R)-s'(x, l; R)s(a, l; R)}{x-a}.
\end{equation}
 \begin{theorem}
 For any $a>0$ and any piecewise continuously differentiable function $\phi(x)$ on $(0, \infty)$, we have
 \begin{equation}
\lim_{R\to\infty}\int^{\infty}_{0} \phi(x)\int^{R}_{0}s(x, l;
r)s(a, l; r)dr\,dx=\frac12[\phi(a^+)+\phi(a^-)],
\end{equation}
provided that the integrals
 \begin{equation}
\int^{1}_{0}x^{-\frac{1}{4}}|\phi(x)|dx\qquad and \qquad
\int^{\infty}_{1}x^{-\frac{3}{4}}|\phi(x)|dx
\end{equation}
are convergent.
 \end{theorem}
\begin{proof} From formulas (4.3) and (4.4), for any $b>\max\{a,
1\}$ there exists a number $M_1>0$ such that
 \begin{equation*}
 \begin{split}
&|s(x, l; R) s'(a, l; R)|\leq M_1x^{-\frac{1}{4}},\\
&|s'(x, l; R) s(a, l; R)|\leq M_1x^{\frac{1}{4}}
\end{split}
\end{equation*}
for $R\geq 1$ and $x\geq b$. Hence, it follows from (4.7) that
 \begin{equation*}
\int^{\infty}_b| \delta_R(x, a)\phi(x)|dx\leq 2M_1\int^\infty_b
\frac{x^{\frac14}|\phi(x)|}{x-a}dx\leq
M_2\int^{\infty}_{b}x^{-\frac{3}{4}}|\phi(x)|dx,
\end{equation*}
where $M_2=2bM_1/(b-a)$. By hypothesis, the last integral is
convergent; so for any $\varepsilon>0$ there is a number $c>b$
such that
 \begin{equation}
\int^{\infty}_c| \delta_R(x, a)\phi(x)|dx<
\frac{\varepsilon}{2}\qquad \qquad\textup{for all}\; R\geq 1.
\end{equation}

To prove that there exists a number $d>0$ such that
 \begin{equation}
\int^{d}_0| \delta_R(x, a)\phi(x)|dx< \frac{\varepsilon}{2}\qquad
\qquad \textup{for all}\;R\geq 1,
\end{equation}
we first need to demonstrate that
 \begin{equation}
k^{\frac12}s(k^2, l; r)=O(1)      \qquad \qquad \textup{as}\;r\to
\infty,
\end{equation}
 \begin{equation}
s'(k^2, l; r)=O(1)      \qquad \qquad \textup{as}\;r\to \infty,
\end{equation}
uniformly for all sufficiently small $k\geq0$. (Recall:
$k=\sqrt{x}$.) This can be done by considering two separate cases:
(i) $kr\to\infty$, and (ii) $kr$ bounded. In case (i), we first
make the change of variable $\rho =kr$ and set
$\omega(\rho)=y(\rho/k)=y(r)$ so that equation (4.1) becomes
  \begin{equation}
\frac{d^2\omega}{d\rho^2}+\biggl(
1+\frac{2}{k\rho}-\frac{l(l+1)}{\rho^2}\biggr)\omega=0,
\end{equation}
and then apply the Liouville-Green transformation given in
\cite[p.196]{Olver} with  $f(\rho)=1+(2/k\rho)$ and
$g(\rho)=l(l+1)/\rho^2$. The result is that equation (4.14) has a
solution $\omega_1(k; \rho)$ such that
 \begin{equation}
\omega_1(k; \rho)\sim \biggl(
1+\frac{2}{k\rho}\biggr)^{\hspace{-0.5mm}-\frac14}\sin\biggl\{\int\biggl(1+\frac{2}{k\rho}\biggr)^{\hspace{-0.5mm}\frac12}
d\rho \biggr\}  \qquad \qquad \textup{as} \;\rho\to\infty
\end{equation}
and
 \begin{equation}
\frac{d}{d\rho}\omega_1(k; \rho)\sim \biggl(
1+\frac{2}{k\rho}\biggr)^{\hspace{-0.5mm}\frac14}\cos\biggl\{\int\biggl(1+\frac{2}{k\rho}\biggr)^{\hspace{-0.5mm}\frac12}
d\rho \biggr\}  \qquad \qquad \textup{as}\; \rho\to\infty.
\end{equation}
With a suitable choice of the integration constant  and for fixed
$k>0$, we have
 \begin{equation*}
\int \biggl( 1+\frac{2}{k\rho}\biggr)^{\frac12}d\rho=\zeta^*(k, l;
\rho)+O\biggl(\frac1\rho\biggr)  \qquad \qquad \textup{as}\;
\rho\to\infty,
\end{equation*}
where
 \begin{equation*}
\zeta^*(k, l; \rho)=\rho+\frac1k \ln 2\rho -\frac{l\pi}{2}+\arg
\Gamma\biggl(l+1-\frac{i}{k}\biggr)
\end{equation*}
which is exactly equal to the function $\zeta(k^2, l; r)$ given in
(4.5). For fixed $k>0$, we can compare the behavior of $s(k^2, l;
r)$ given in (4.3) and that given in (4.15). The conclusion is
 \begin{equation}
s(k^2, l; r)=\frac{1}{\sqrt{\pi k}} \omega_1 (k; \rho),
\end{equation}
from which we also obtain
 \begin{equation}
s'(k^2, l; r)=\sqrt{\frac{k}{\pi}}\frac{d\omega_1}{d\rho}.
\end{equation}
Since $\omega_1 (k; \rho)=O(1)$ and $\omega'_1 (k;
\rho)=O(k^{-\frac14})$ for all small $k$ and large $\rho$ on
account of (4.15) and (4.16), the order estimates in (4.12) and
(4.13) are established.

In case (ii), we first recall the function
 \begin{equation*}
f(k^2, l; r)=\frac{(i/k)^{l+1}}{\Gamma(2l+2)}M_{i/k,
l+\frac12}(-2ikr),
\end{equation*}
where $M_{\kappa, \lambda}(z)$ is  a Whittaker function; see
Seaton \cite[eqs. (14) \& (22)]{Seaton}. This function is related
to the Coulomb wave function $s(k^2, l; r)$ via
 \begin{equation*}
s(k^2, l; r)=\biggl[\frac{A(k^2,
l)}{2(1-e^{-2\pi/k})}\biggr]^\frac12 f(k^2, l; r),
\end{equation*}
where $A(k^2, l)$ is a polynomial of degree $l$ in $k^2$; see
\cite[eq.(114)]{Seaton}. In view of the convergent expansion
\cite[$\S$ 7, eq.(16)]{Buchholz}
 \begin{equation*}
M_{\kappa,
\lambda}(z)=\Gamma(2\lambda+1)2^{2\lambda}z^{\lambda+\frac12}\sum^\infty_{n=0}p^{(2\lambda)}_n(z)\frac{J_{2\lambda+n}(2\sqrt{z\kappa})}{(2\sqrt{z\kappa})^{2\lambda+n}},
\end{equation*}
where the $p^{(2\lambda)}_n(z)$ are polynomials in $z^2$, we have
for bounded $kr$
\begin{equation*}
f(k^2, l;
r)=C(kr)r^\frac14\cos\biggl(2\sqrt{2r}-\frac{\pi}{2}(2l+1)-\frac14\pi\biggr)+O\biggl(\frac{1}{r^{1/4}}\biggr)
\end{equation*}
as $r\to\infty$, where $C(kr)$ is a polynomial of $kr$. Hence,
\begin{equation}
\begin{split}
s(k^2, l; r)&=C^*(kr)k^{-\frac14}\biggl[\frac{A(k^2,
l)}{2(1-e^{-2\pi/k})}\biggr]^\frac12\cos\biggl(2\sqrt{2r}-\frac{\pi}{2}(2l+1)-\frac14\pi\biggr)\\
&\hspace*{3.8mm}+O\biggl(\frac{1}{r^{1/4}}\biggr)
\end{split}
\end{equation}
and
\begin{equation}
\begin{split}
s'(k^2, l; r)&=-C(kr)r^{-\frac14}\biggl[\frac{A(k^2,
l)}{(1-e^{-2\pi/k})}\biggr]^\frac12\sin\biggl(2\sqrt{2r}-\frac{\pi}{2}(2l+1)-\frac14\pi\biggr)\\
&\hspace*{3.8mm}+O\biggl(\frac{1}{r^{3/4}}\biggr),
\end{split}
\end{equation}
where $C^*(kr)=(kr)^{\frac14}C(kr)$, again proving (4.12) and
(4.13).

From (4.12) and (4.13), it follows that there are constants
$0<\rho<\min\{a, 1\}$, $R_0>0$ and $N_1>0$ such that
\begin{equation*}
|x^{\frac14}s(x, l; r)|\leq N_1 \qquad \textup{and}\qquad |s'(x,
l; r)|\leq N_1
\end{equation*}
for all $r\geq R_0$ and $0\leq x\leq \rho$. Furthermore, by (4.7),
\begin{equation*}
|\delta_R(x, a)|\leq
\frac{N^2_1}{|x-a|}(x^{-\frac14}+a^{-\frac14})\leq 2
N^2_1\frac{x^{-\frac14}}{|x-a|}
\end{equation*}
for $0<x\leq\rho$ and $R\geq R_0$, and
\begin{equation*}
\int^\rho_0 |\delta_R(x, a)\phi(x)|dx\leq N_2 \int^\rho_0
x^{-\frac14}|\phi(x)|dx,
\end{equation*}
where $N_2=2N_1^2/(a-\rho)$. By hypothesis, the last integral is
convergent, thus establishing (4.11).

Let us now consider the case when $x$ lies in the interval $(d,
c)$. From (4.3), (4.4) and (4.7), we have
 \begin{equation}
 \begin{split}
\pi(x-a)\delta_R(x, a)&= \biggl(\frac{a}{x}\biggr)^{\hspace{-0.5mm}\frac14} \sin\zeta(x, l; R)\cos\zeta(a, l; R)  \\
&\hspace*{5mm}-\biggl(\frac{x}{a}\biggr)^{\hspace{-0.5mm}\frac14}
\cos\zeta(x, l; R)\sin\zeta(a, l; R)+\varepsilon(x, a; R),
\end{split}
\end{equation}
where $\varepsilon(x, a; R)/(x-a)$ is continuous in $(0, \infty)$
and $\varepsilon(x, a; R)=O(1/R)$ as $R\to\infty$ uniformly for
$x\in(d, c)$. As a consequence, we obtain
\begin{equation}
\lim_{R\to \infty}\int^c_d \frac{\varepsilon(x, a;
R)}{x-a}\phi(x)dx=0.
\end{equation}

For any $A$ and $B$ satisfying $d\leq A<B\leq c$, an integration
by parts yields
 \begin{equation*}
 \begin{split}
\int^B_A \sin \zeta(k^2, l; R)dk&=-\frac{\cos\zeta(k^2, l; R)}{\zeta_k(k^2, l; R)} \biggl|^B_A  \\
&\hspace*{4mm}+\int^B_A\frac{\partial}{\partial
k}\biggl(\frac{1}{\zeta_k(k^2, l; R)}\biggr)\cos\zeta(k^2, l;
R)dk,
\end{split}
\end{equation*}
where $\zeta_k(k^2, l; R)$ denotes the derivative of $\zeta(k^2,
l; R)$ with respect to $k$. From (4.5), it is readily seen that
for $d\leq k^2\leq c$,
\begin{equation*}
\zeta_k(k^2, l; R)\to \infty
\qquad\textup{and}\qquad\frac{\partial}{\partial
k}\biggl(\frac{1}{\zeta_k(k^2, l; R)}\biggr)\to0
\end{equation*}
as $R\to\infty$. Hence,
\begin{equation}
\lim_{R\to \infty}\int^B_A \sin\zeta(k^2, l; R)dk=0.
\end{equation}
Similarly, we also have
\begin{equation}
\lim_{R\to \infty}\int^B_A \cos\zeta(k^2, l; R)dk=0.
\end{equation}
Let $\eta>0$ be an arbitrary number such that $d<a-\eta<a+\eta<c$.
A combination of (4.23), (4.24) and the generalized
Riemann-Lebesgue lemma gives
\begin{equation}
\lim_{R\to
\infty}\biggl(\int^{a-\eta}_d+\int^c_{a+\eta}\biggr)\biggl(\frac{a}{x}\biggr)^{\hspace{-0.5mm}\frac14}\frac{\phi(x)}{x-a}\sin\zeta(x,
l; R)dx=0
\end{equation}
and
\begin{equation}
\lim_{R\to
\infty}\biggl(\int^{a-\eta}_d+\int^c_{a+\eta}\biggr)\biggl(\frac{x}{a}\biggr)^{\hspace{-0.5mm}\frac14}\frac{\phi(x)}{x-a}\cos\zeta(x,
l; R)dx=0.
\end{equation}
By the same reasoning, we have
\begin{equation}
\lim_{R\to
\infty}\int^{a+\eta}_{a-\eta}\biggl[\biggl(\frac{a}{x}\biggr)^{\hspace{-0.5mm}\frac14}-1\biggr]\frac{\phi(x)}{x-a}\sin\zeta(x,
l; R)dx=0
\end{equation}
and
\begin{equation}
\lim_{R\to
\infty}\int^{a+\eta}_{a-\eta}\biggl[\biggl(\frac{x}{a}\biggr)^{\hspace{-0.5mm}\frac14}-1\biggr]\frac{\phi(x)}{x-a}\cos\zeta(x,
l; R)dx=0.
\end{equation}
From (4.21), (4.22) and $(4.25)-(4.28)$, it follows that
 \begin{equation}
 \begin{split}
\lim_{R\to \infty}\int^{c}_{d}&\delta_R(x, a)\phi(x)dx\\
&\hspace*{-1mm}=\lim_{R\to \infty}\int^{a+\eta}_{a-\eta}\frac{\sin\{\zeta(x, l; R)-\zeta(a, l; R)\}}{x-a}\phi(x)dx  \\
&\hspace*{-1mm}=\lim_{R\to
\infty}\biggl(\int^{a}_{a-\eta}+\int^{a+\eta}_{a}\biggr)\frac{\sin\{\zeta(x,
l; R)-\zeta(a, l; R)\}}{x-a}\phi(x)dx.
\end{split}
\end{equation}

Since $\phi(x)$ is piecewise continuously differentiable in $(0,
\infty)$, it is continuously differentiable in $(a-\eta, a)$ for
sufficiently small $\eta>0$, and $\phi(x)-\phi(a^-)/(x-a)$ is
integrable on $(a-\eta, a)$. By (4.23), (4.24) and the generalized
Riemann-Lebesgue lemma, we have
\begin{equation}
\lim_{R\to
\infty}\int^{a}_{a-\eta}\frac{\phi(x)-\phi(a^-)}{x-a}\sin
\{\zeta(x, l; R)-\zeta(a, l; R)\}dx=0,
\end{equation}
or equivalently
\begin{equation*}
\begin{split}
\lim_{R\to \infty}\int^{a}_{a-\eta}&\frac{\sin \{\zeta(x, l;
R)-\zeta(a, l; R)\}}{\pi(x-a)}\phi(x)dx\\
&\hspace*{18.7mm}=\phi(a^-)\lim_{R\to
\infty}\int^{a}_{a-\eta}\frac{\sin \{\zeta(x, l; R)-\zeta(a, l;
R)\}}{\pi(x-a)} dx.
\end{split}
\end{equation*}

To obtain the value of the limit on the right-hand side of the
last equation, we shall use the Cauchy residue theorem. Let
\begin{equation}
\begin{split}
\overline{\zeta}(x, R)&=\sqrt{x}R+\frac{1}{\sqrt{x}}\ln R,\\
 \theta(x,
l)&=\frac{1}{\sqrt{x}}\ln(2\sqrt{x})-\frac{l\pi}{2}+\arg
\Gamma\biggl(l+1-\frac{i}{\sqrt{x}}\biggr)
\end{split}
\end{equation}
so that
\begin{equation}
\zeta(x, l; R)=\overline{\zeta}(x; R)+ \theta(x, l).
\end{equation}
Furthermore, let $\Gamma$ denote the positively oriented closed
contour depicted in Figure 1 below. It consists of a horizontal
line segment $\Gamma_3$, two vertical line segments $\Gamma_2$ and
$\Gamma_4$, a quarter-circle $\Sigma$ centered at $z=a$ with
radius $r$, and the interval $\Gamma_1$ on the positive
real-axis. The entire region bounded by $\Gamma$ lies in the first
quadrant $\{ z\in \mathbb{C}: \textup{Re}\; z>0$ and $
\textup{Im}\;z\geq 0\}$. Consider the complex-value function
\begin{equation}
F_R(z)=\frac{e^{i[\overline{\zeta}(z, R)-\overline{\zeta}(a,
R)]}}{z-a}.
\end{equation}
Since $\overline{\zeta}(z, R)$ is analytic in the right
half-plane, by Cauchy's theorem
\begin{equation}
\int_\Gamma F_R(z)dz=0.
\end{equation}
By Cauchy's residue theorem, we also have
\begin{equation}
\lim_{r\to 0^+}\int_{\Sigma} F_R(z)dz=-i \pi/2,
\end{equation}
where $r$ is the radius of the quarter-circle $\Sigma$.

\begin{figure}[htp]
\begin{center}
  \includegraphics[width=3in]{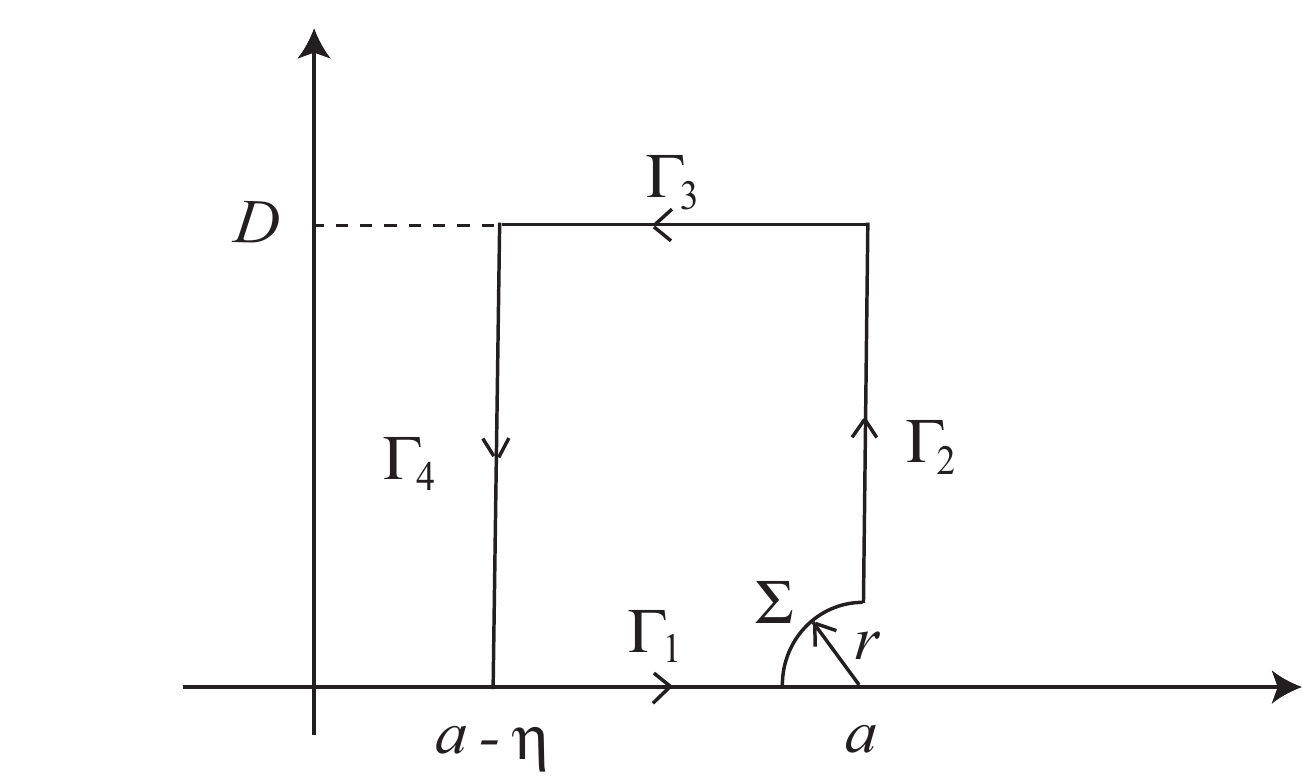}\\
  \caption{Contour $\Gamma$.}\label{AIMS}
  \end{center}
\end{figure}

For $z\in \Gamma_3$, we write $z=u+iD$. Since $\overline{\zeta}(a,
R)$ is real, a simple estimation gives
\begin{equation}
|F_R(z)|\leq \frac{1}{D}e^{-\textup{Im}\;\overline{\zeta}(z, R)},
\qquad \qquad z\in \Gamma_3.
\end{equation}
From (4.31), it is readily seen that $\lim_{R\to
\infty}\textup{Im}\;\overline{\zeta}(z, R)=+\infty$ uniformly for
$z\in \Gamma_3$. Hence, $\lim_{R\to
\infty}e^{-\textup{Im}\;\overline{\zeta}(z, R)}=0$ uniformly for
$z\in \Gamma_3$. From (4.36), it follows that
\begin{equation}
\lim_{R\to \infty}\int_{\Gamma_3}F_R(z)dz=0.
\end{equation}

For  $z\in \Gamma_4$, we write $z=a-\eta+iv$. Clearly
\begin{equation}
|F_R(z)|\leq \frac{1}{\eta}e^{-\textup{Im}\;\overline{\zeta}(z,
R)}, \qquad \qquad z\in \Gamma_4.
\end{equation}
Let $\sigma_1>0$ be any small number. From (4.31), we have
$\lim_{R\to \infty}\textup{Im}\;\overline{\zeta}(z, R)=+\infty$
uniformly for $z\in \Gamma_4$ and $\textup{Im}\; z=v \geq
\sigma_1$. Thus, $\lim_{R\to
\infty}e^{-\textup{Im}\;\overline{\zeta}(z, R)}=0$ uniformly for
$z\in \Gamma_4\cap\{z:\textup{Im}\; z\geq \sigma_1 \}$ and
\begin{equation*}
\lim_{R\to
\infty}\int_{\Gamma_4}e^{-\textup{Im}\;\overline{\zeta}(z, R)}dz
=\lim_{R\to \infty}\int^{a-\eta+i
\sigma_1}_{a-\eta}e^{-\textup{Im}\;\overline{\zeta}(z, R)}dz.
\end{equation*}
From (4.31), it also follows that there is a constant $M_5>0$ such
that $e^{-\textup{Im}\;\overline{\zeta}(z, R)}\leq M_5$. Hence
\begin{equation*}
\biggl|\lim_{R\to
\infty}\int_{\Gamma_4}e^{-\textup{Im}\;\overline{\zeta}(z, R)}dz
\biggr|\leq M_5\sigma_1.
\end{equation*}
Since $\sigma_1$ can be arbitrarily small, we obtain
\begin{equation}
\lim_{R\to
\infty}\int_{\Gamma_4}e^{-\textup{Im}\;\overline{\zeta}(z, R)}dz
=0.
\end{equation}
Coupling (4.38) and (4.39) gives
\begin{equation}
\lim_{R\to \infty}\int_{\Gamma_4}F_R(z)dz =0.
\end{equation}
In a similar manner, one can establish
\begin{equation}
\lim_{R\to \infty}\int_{\Gamma_2}F_R(z)dz =0.
\end{equation}
By a combination of (4.34), (4.35), (4.37) and $(4.40)-(4.41)$, we
obtain
\begin{equation}
\begin{split}
\lim_{R\to \infty}\int^a_{a-\eta}&\frac{\sin \{\overline{\zeta}(x, R)-\overline{\zeta}(a, R)\}}{\pi(x-a)}dx\\
&\hspace*{16.2mm}=-\frac1\pi \textup{Im}\;\biggl\{\lim_{R\to
\infty}\lim_{r\to 0^+}\int_{\Sigma}F_R(z)dz\biggr\}=\frac12.
\end{split}
\end{equation}
Using (4.42), we are now ready to handle the limit on the
right-hand side of the equation following (4.30).

For simplicity, let us write $\theta(x)=\theta(x, l)$; cf. (4.31).
By the mean-value theorem,
$\cos[\theta(x)-\theta(a)]=1+O\{(x-a)^2\}$ for $x$ near $a$.
Hence, by the  generalized Riemann-Lebesgue lemma, we have from
(4.42)
\begin{equation}
\lim_{R\to \infty}\int^{a}_{a-\eta}
\frac{\sin\{\overline{\zeta}(x, R)-\overline{\zeta}(a,
R)\}\cos[\theta(x)-\theta(a)]}{\pi(x-a)}dx=\frac12,
\end{equation}
and also from (4.23) and (4.24)
\begin{equation}
\lim_{R\to \infty}\int^{a}_{a-\eta}
\frac{\cos\{\overline{\zeta}(x, R)-\overline{\zeta}(a,
R)\}\sin[\theta(x)-\theta(a)]}{\pi(x-a)}dx=0.
\end{equation}
Upon using an addition formula, it follows from (4.43) and (4.44)
that
\begin{equation}
\lim_{R\to \infty}\int^{a}_{a-\eta} \frac{\sin\{\zeta(x, l;
R)-\zeta(a, l; R)\}}{\pi(x-a)}dx=\frac12.
\end{equation}
Coupling (4.30) and (4.45), we obtain
\begin{equation}
\lim_{R\to \infty}\int^{a}_{a-\eta} \frac{\sin\{\zeta(x, l;
R)-\zeta(a, l; R)\}}{\pi(x-a)}\phi(x)dx=\frac12\phi(a^-).
\end{equation}
In a similar manner, we  also have
\begin{equation}
\lim_{R\to \infty}\int_{a}^{a+\eta} \frac{\sin\{\zeta(x, l;
R)-\zeta(a, l; R)\}}{\pi(x-a)}\phi(x)dx=\frac12\phi(a^+).
\end{equation}
A combination of (4.29), (4.46) and (4.47) yields
\begin{equation}
\lim_{R\to \infty}\int^{c}_{d} \delta_R(x,
a)\phi(x)dx=\frac12[\phi(a^-)+\phi(a^+)].
\end{equation}
Since $\varepsilon$ in (4.10) and (4.11) can be arbitrarily small,
(4.8) now follows from (4.48). This completes the proof of the
theorem.
\end{proof}

\section{Airy and parabolic cylinder functions}

 We now turn our attention to the integral representation (1.7). Here, the interval of concern is the whole real line. However, the argument for this result remains similar to that for Theorems 1 \& 2, and we will keep it brief. As before, we let $b$ and $\eta$ be positive numbers such that $b>\max\{1, |a|\}$. The Airy function $\textup{Ai}(t-x)$ is a solution of the equation
 \begin{equation}
\frac{d^2y}{dt^2}+(x-t)y=0, \qquad \qquad -\infty<t<\infty,
\end{equation}
and satisfies
 \begin{equation}
\textup{Ai} (\infty)=\textup{Ai}' (\infty)=0,
\end{equation}
 \begin{equation}
\textup{Ai} (-t-x)=\frac{1}{\sqrt{\pi}}(t+x)^{-1/4}[\sin\zeta(x,
t)+\varepsilon_1(x, t)],
\end{equation}
and
 \begin{equation}
\textup{Ai}' (-t-x)=-\frac{1}{\sqrt{\pi}}(x+t)^{1/4}[\cos\zeta(x,
t)+\varepsilon_2(x, t)],
\end{equation}
where $\varepsilon_1(x, t)=O(t^{-3/2})$ and $\varepsilon_2(x,
t)=O(t^{-3/2})$, as $t\to \infty$, uniformly for $x>-b$, and where
 \begin{equation}
\zeta(x, t)=\frac{2}{3}(t+x)^{3/2}+\frac{\pi}{4}.
\end{equation}
Define
 \begin{equation}
\delta_R(x, a)=\int^\infty_{-R}\textup{Ai}(t-x)\textup{Ai}(t-a)dt.
\end{equation}
From (5.2), we have
 \begin{equation}
\delta_R(x,
a)=\frac{\textup{Ai}(-R-x)\textup{Ai}'(-R-a)-\textup{Ai}(-R-a)\textup{Ai}'(-R-x)}{x-a};
\end{equation}
see (3.6).

  \begin{theorem} For any $a\in\mathbb{R}$ and any piecewise continuously differentiable function $\phi(x)$ on $(-\infty, \infty)$, we have
 \begin{equation}
\lim_{R\to\infty}\int^\infty_{-\infty}\phi(x)\int^\infty_{-R}\textup{Ai}
(t-x)\textup{Ai}(t-a)dt\,dx=\frac12[\phi(a^-)+\phi(a^+)],
\end{equation}
provided that the two integrals
 \begin{equation}
\int^{-1}_{-\infty}|x|^{-\frac34}|\phi(x)|dx\qquad and \qquad
\int^\infty_{1}x^{-\frac34}|\phi(x)|dx
\end{equation}
are convergent.
  \end{theorem}
 \begin{proof} Using the asymptotic formulas (5.3) and (5.4), one
can show from (5.7) that there are positive constants $M_1$ and
$M_2$ such that
 \begin{equation*}
\int^\infty_{b}|\delta_R(x, a)\phi(x)|dx\leq M_1\int^\infty_b
x^{-\frac34}|\phi(x)|dx
\end{equation*}
and
 \begin{equation*}
\int^{-b}_{-\infty}|\delta_R(x, a)\phi(x)|dx\leq
M_2\int_{-\infty}^{-b} |x|^{-\frac34}|\phi(x)|dx
\end{equation*}
for $R>2|a|+1$. On account of the convergence of the two integrals
in (5.9), for any $\varepsilon>0$ there is a constant $c>b$ such
that
 \begin{equation}
\biggl(\int^{-c}_{-\infty}+\int^{\infty}_{c}\biggr)|\delta_R(x,
a)\phi(x)|dx< \varepsilon
\end{equation}
for all $R>2|a|+1$.

In view of the asymptotic formulas (5.3) and (5.4), equation (5.7)
also gives
 \begin{equation}
 \begin{split}
\pi(x-a)\delta_R(x, a)=&\biggl( \frac{R+a}{R+x}\biggr)^{\hspace{-0.5mm}\frac14}\sin\zeta(x, R)\cos \zeta(a, R)\\
&-\biggl( \frac{R+x}{R+a}\biggr)^{\hspace{-0.5mm}\frac14}\cos\zeta(x, R)\sin \zeta(a, R)+\varepsilon(x, a ; R),
\end{split}
\end{equation}
where $\varepsilon(x, a ; R)=O(R^{-5/4})$, as $R\to \infty$,
uniformly for $x, a \in (-b, b)$. In a manner similar to (4.29),
by using the generalized Riemann-Lebesgue lemma we have
 \begin{equation}
\lim_{R\to \infty}\int^{c}_{-c}\delta_R(x, a)\phi (x)dx=\lim_{R\to
\infty}\int^{a+\eta}_{a-\eta}\frac{\sin\{\zeta(x, R)-\zeta(a,
R)\}}{\pi (x-a)}\phi(x)dx,
\end{equation}
for any piecewise continuously differentiable function $\phi(x)$
and any $\eta>0$. Since $\zeta(x, R)-\zeta(a,
R)=\sqrt{R}(x-a)+O(1/\sqrt{R})$ for large $R$ and bounded $x$ and
$a$, we also have $\sin \{\zeta(x, R)-\zeta(a, R)\}=\sin
(\sqrt{R}(x-a))+O(1/\sqrt{R})$ as $R\to \infty$ for bounded $x$
and $a$. Thus, equation (5.12) gives
 \begin{equation}
\lim_{R\to \infty}\int^{c}_{-c}\delta_R(x, a)\phi (x)dx=\lim_{R\to
\infty}\int^{a+\eta}_{a-\eta}\frac{\sin\sqrt{R}(x-a)}{\pi
(x-a)}\phi(x)dx.
\end{equation}
By Jordan's theorem on the Dirichlet kernel \cite[p.473]{Apostol},
the value of the last limit is $[\phi(a^-)+\phi(a^+)]/2$. The
final result (5.8) now follows from (5.10) and (5.13).
\end{proof}

The parabolic cylinder function $W(a, x)$ is a solution of Weber's
equation
 \begin{equation}
\frac{d^2y}{dx^2}+\biggl(\frac14 x^2-a\biggr)y=0
\end{equation}
with boundary conditions
 \begin{equation}
y(x)=\sqrt{\frac{2k_a}{x}}\biggl[\cos \zeta(a,
x)+O\biggl(\frac1x\biggr)\biggr],\qquad  \qquad x\to \infty,
\end{equation}
and
 \begin{equation}
y(-x)=\sqrt{\frac{2}{k_ax}}\biggl[\sin \zeta(a,
x)+O\biggl(\frac1x\biggr)\biggr],\qquad  \qquad x\to \infty,
\end{equation}
where $k_a=\sqrt{1+e^{2\pi a}}-e^{\pi a}$ and
 \begin{equation}
\zeta(a, x)=\frac14 x^2-a \ln x +\frac12 \arg
\Gamma\biggl(\frac12+ia\biggr)+\frac\pi4;
\end{equation}
see \cite[p.693]{AS}. The derivative of this function has the
behavior
\begin{equation}
W'(a, x)=-\sqrt{\frac{k_ax}{2}}\biggl[\sin \zeta(a,
x)+O\biggl(\frac1x\biggr)\biggr],\qquad  \qquad x\to \infty,
\end{equation}
and
\begin{equation}
W'(a, -x)=-\sqrt{\frac{x}{2k_a}}\biggl[\cos \zeta(a,
x)+O\biggl(\frac1x\biggr)\biggr],\qquad  \qquad x\to \infty.
\end{equation}
This function is also related to the parabolic cylinder function
$U(a, x)$ via the connection formulas \cite[p.693]{AS}
\begin{equation}
W(a, x)=(2k_a)^\frac12e^{\frac14\pi
a}\,\textup{Re}\;\bigl\{e^{i(\frac12\phi_2+\frac18\pi)}U(ia,
xe^{-\frac14\pi i})\bigr\},
\end{equation}
\begin{equation}
W(a, -x)=(2/k_a)^\frac12e^{\frac14\pi
a}\,\textup{Im}\;\bigl\{e^{i(\frac12\phi_2+\frac18\pi)}U(ia,
xe^{-\frac14\pi i})\bigr\},
\end{equation}
where $x>0$ and $\phi_2=\arg \Gamma(\frac12+ia)$. From the
integral representation \cite[p.208]{Olver}
\begin{equation}
U(a, z)=\frac{e^{-\frac14z^2}}{\Gamma(\frac12+a)}\int^\infty_0
e^{-zs-\frac12s^2} s^{a-\frac12}\,ds, \qquad \qquad
\textup{Re}\; a>-\frac12,
\end{equation}
one can easily show that
\begin{equation*}
|U(ia, xe^{-\frac14\pi i})|\leq
\frac{2^{\frac14}\sqrt{\pi}}{|\Gamma(\frac12+ia)|}x^{-\frac12}.
\end{equation*}
Since
\begin{equation*}
\left|\Gamma({\textstyle\frac12}+ia)\right|=\frac{\sqrt{\pi}}{(\cosh
\pi a)^{1/2}}\sim \sqrt{2\pi}e^{-\frac12\pi a}
\end{equation*}
and
\begin{equation*}
\sqrt{2k_a}\sim e^{-\frac12\pi a}, \qquad  \qquad
\sqrt{\frac{2}{k_a}}\sim 2e^{\frac12\pi a}
\end{equation*}
as $a\to+\infty$, it follows that for large positive $a$, there is
a constant $M_1$ such that
\begin{equation*}
\begin{split}
&|W(a, x)|\leq M_1e^{\frac14\pi a}x^{-\frac12}, \qquad  \qquad x>0,\\
&|W(a, -x)|\leq M_1e^{\frac54\pi a}x^{-\frac12}, \qquad  \qquad
x>0.
\end{split}
\end{equation*}
As $a\to -\infty$, we have
\begin{equation*}
\left|\Gamma({\textstyle\frac12}+ia)\right|\sim \sqrt{2\pi}
e^{\frac12\pi a}= \sqrt{2\pi}e^{-\frac12\pi |a|}
\end{equation*}
and
\begin{equation*}
 \sqrt{2k_a}\sim   \sqrt{2},  \qquad  \qquad  \sqrt{\frac{2}{k_a}}\sim \sqrt{2}.
\end{equation*}
Hence, there exists a constant $M'_1>0$ such that for large
negative $a$,
\begin{equation*}
\begin{split}
&|W(a, x)|\leq M'_1e^{\frac14\pi |a|}x^{-\frac12}, \qquad  \qquad x>0,\\
&|W(a, -x)|\leq M'_1e^{\frac14\pi |a|}x^{-\frac12}, \qquad  \qquad
x>0.
\end{split}
\end{equation*}
By using (5.20), (5.21) and (5.22), it can also be shown that
there are positive constants $M_2$ and $M'_2$ such that for large
positive $a$,
\begin{equation*}
\begin{split}
&|W'(a, x)|\leq M_2e^{\frac14\pi a}x^{\frac12}, \qquad  \qquad x>0,\\
&|W'(a, -x)|\leq M_2e^{\frac54\pi a}x^{\frac12}, \qquad  \qquad
x>0,
\end{split}
\end{equation*}
and for large negative $a$,
\begin{equation*}
\begin{split}
&|W'(a, x)|\leq M'_2e^{\frac14\pi |a|}x^{\frac12}, \qquad  \qquad x>0,\\
&|W'(a, -x)|\leq M'_2e^{\frac14\pi |a|}x^{\frac12}, \qquad  \qquad
x>0.
\end{split}
\end{equation*}

We define
\begin{equation}
\delta_R(a, b)=\int^R_{-R}W(a, x) W(b, x)dx.
\end{equation}
From (5.14), one can derive
\begin{equation*}
\delta_R(a, b)=\frac{W(a, x) W'(b, x)-W'(a, x) W(b,
x)}{b-a}\biggl|^R_{-R}.
\end{equation*}
Using the asymptotic formulas (5.15) and (5.18), we obtain
\begin{equation*}
(b-a)\delta_R(a,
b)=\biggl(\sqrt{k_ak_b}+\frac{1}{\sqrt{k_ak_b}}\biggl)\sin\{\zeta(a,
R)-\zeta(b, R)\}+O\biggl(\frac1R\biggr).
\end{equation*}
By an argument similar to that for Theorem 3, one can establish
the following result.
  \begin{theorem}
For any $a\in\mathbb{R}$ and any piecewise continuously
differentiable function $\phi(x)$ on $(-\infty, +\infty)$, we have
 \begin{equation*}
\lim_{R\to \infty}\int^{+\infty}_{-\infty}\phi (b)\int^R_{-R}W(a,
x)W(b,  x)\,dx \,db=\pi\sqrt{1+e^{2\pi a}}[\phi(a^-)+\phi(a^+)],
\end{equation*}
provided that two integrals
 \begin{equation*}
\int^{-1}_{-\infty}|\phi(x)|e^{\frac14 \pi |x|}\frac{dx}{|x|}
\qquad and\qquad \int^{\infty}_{1}|\phi(x)|e^{\frac54 \pi x}
\frac{dx}{x}
\end{equation*}
are convergent.
  \end{theorem}

\section{Series representations}
To prove the representations in $(1.8)-(1.11)$, we only need to
recall  some expansion theorems concerning orthogonal polynomials
(functions). For instance, in the case of Legendre polynomials
$P_n(x)$, Theorem 1 in \cite[p.55]{Lebedev} can be stated in the
following form.

\begin{theorem}  Let $f(x)$ be a piecewise continuously
differentiable function in $(-1, 1)$, and put
\begin{equation}
\delta_n(t,x)=\sum^n_{k=0}\biggl(k+\frac12\biggr)P_k(t)P_k(x).
\end{equation}
If the integral
\begin{equation*}
\int^{1}_{-1} f^2(x)\,dx
\end{equation*}
is finite, then
\end{theorem}
\begin{equation}
\lim_{n\to\infty}\int^{1}_{-1}\delta_n(t,
x)f(t)\,dt=\frac12[f(x^+)+f(x^-)].
\end{equation}

The statement in (6.2) is equivalent to that in (1.8); i.e., the
finite sum in (6.1) defines a delta sequence. In a similar manner,
one can restate Theorems 2 and 3 in \cite[p.71 and p.88]{Lebedev}
as follows.

\begin{theorem}  Let $f(x)$ be a piecewise continuously
differentiable function in $(-\infty, \infty)$, and put
\begin{equation}
\delta_n(t, x)=e^{-t^2}\sum^n_{k=0}\frac{1}{2^k
k!\sqrt{\pi}}H_k(t)H_k(x).
\end{equation}
If the integral
\begin{equation*}
\int^\infty_{-\infty}e^{-x^2}f^2(x)\,dx
\end{equation*}
is finite, then
\begin{equation}
\lim_{n\to\infty}\int^\infty_{-\infty}\delta_n(t, x)f(t)\,
dt=\frac12[f(x^+)+f(x^-)].
\end{equation}
\end{theorem}

\begin{theorem} Let $f(x)$ be a piecewise continuously
differentiable function in $(0, \infty)$, and put
\begin{equation}
\delta_n(t,
x)=e^{-t}t^\alpha\sum^n_{k=0}\frac{k!}{\Gamma(k+\alpha+1)}L^{(\alpha)}_k(t)L^{(\alpha)}_k(x).
\end{equation}
If the integral
\begin{equation*}
\int^\infty_{0}e^{-t}t^\alpha f^2(t)\,dt, \qquad\qquad \alpha>-1,
\end{equation*}
is finite, then
\begin{equation}
\lim_{n\to\infty}\int^\infty_{0}\delta_n(t, x)f(t)\,
dt=\frac12[f(x^+)+f(x^-)].
\end{equation}
\end{theorem}

To demonstrate (1.11), we recall the Laplace series expansion
\begin{equation}
f(\theta_1,
\phi_1)=\sum^\infty_{k=0}\frac{2k+1}{4\pi}\int^\pi_{-\pi}\int^\pi_0f(\theta_2,
\phi_2)P_k(\cos \gamma)\sin\theta_2  d  \theta_2 d \phi_2,
\end{equation}
where $\cos \gamma=\cos \theta_1\cos \theta_2-\sin \theta_1\sin
\theta_2\cos(\phi_1-\phi_2)$ and $P_k(x)$ is a Legendre
polynomial; see \cite[p.147]{Hobson}. By the addition formula
\cite[p.797]{AW}
\begin{equation}
P_k(\cos \gamma)=\frac{4\pi}{2k+1} \sum^k_{l=-k} Y_{kl}(\theta_1,
\phi_1)Y^*_{kl}(\theta_2, \phi_2),
\end{equation}
we can rewrite (6.7) in the form
\begin{equation}
f(\theta_1,
\phi_1)=\sum^\infty_{k=0}\int^\pi_{-\pi}\int^\pi_0f(\theta_2,
\phi_2)\sin \theta_2  \sum^k_{l=-k}  Y_{kl}(\theta_1,
\phi_1)Y^*_{kl}(\theta_2, \phi_2)d \theta_2 d \phi_2.
\end{equation}
For $f\in C([0, \pi]\times[-\pi, \pi])$, the series on the right
converges pointwise to the function on the left; see
\cite[p.344]{Hobson}. This result can be expressed as follows.

\begin{theorem} Let $f(\theta_1, \phi_1)$ be a continuous
function on $[0, \pi]\times[-\pi, \pi]$, and put
\begin{equation}
\delta_n(\theta_1, \theta_2)\delta_n(\phi_1,
\phi_2):=\sin\theta_2\sum^n_{k=0}\sum^k_{l=-k}Y_{kl}(\theta,
\phi_1)Y^*_{kl}(\theta_2, \phi_2).
\end{equation}
Then, we have
\end{theorem}
\begin{equation}
f(\theta_1,
\phi_1)=\lim_{n\to\infty}\int^\pi_{-\pi}\int^\pi_{0}\delta_n(\theta_1,
\theta_2)\delta_n(\phi_1, \phi_2)f(\theta_2, \phi_2)\,d\theta_2 \,
d\phi_2.
\end{equation}

Equation (6.11) is equivalent to saying that $\delta_n(\theta_1,
\theta_2)\delta_n(\phi_1, \phi_2)$ is a delta sequence of
$\delta(\theta_1-\theta_2)\delta(\phi_1-\phi_2)$ for $\theta_1,
\theta_2\in [0, \pi]$ and $\phi_1, \phi_2\in [-\pi, \pi]$. The use
of the identity $\delta(\cos\theta_1, \cos\theta_2)=\frac {1}{\sin
\theta_2}\delta(\theta_1-\theta_2)$ gives (1.11); see
\cite[p.49]{Kanwal}.

\section*{Acknowledgements} The authors would like to thank Professor W. Y. Qiu of Fudan University for many  helpful discussions. His unfailing assistance is greatly appreciated.


\begin{thebibliography}{99}

\bibitem{AS}(0167642) M. Abramowitz and I. A. Stegun (eds.), ``Handbook of
Mathematical Functions," Appl. Math. Ser. No. 55, National Bureau
of Standards, Washington, D.C.,  1964. (Reprinted by Dover, New
York, 1965).

\bibitem{Apostol}(0087718) T. M. Apostol, ``{Mathematical Analysis}," Addison-Wesley,
Reading, MA, 1957.


\bibitem{AW}(1810939) G. B. Arfken and H. J. Weber, ``{Mathematical Methods for
Physicists} (6th ed.)," Elsevier, Oxford, 2005.


\bibitem{Buchholz}(0240343) H. Buchholz, ``{The Confluent Hypergeometric Function},"
Springer-Verlag, Berlin and New York, 1969.




\bibitem{GS}(0166596) I. M. Gel'fand and G. E. Shilov, ``{Generalized
Functions}," Academic Press, New York and London, 1964.

\bibitem{Hobson}(0064922) E. W. Hobson, ``{The Theory of Spherical and Ellipsoidal
Harmonics} (2nd ed.)," Chelsea Publishing Co., New York, 1955.

\bibitem{Jones}(0217534) D. S. Jones, ``{Generalized Functions}," McGraw-Hill,
London, 1966.

\bibitem{Kanwal}(1604296) R. P. Kanwal, ``{Generalized Functions: Theory and
Techniques} (2nd ed.)," Birkh\"{a}user, Boston, 1998.


\bibitem{Lebedev}(0174795) N. N. Lebedev, ``{Special Functions and Their
Applications}," Prentice-Hall, London, 1965.


\bibitem{Lighthill}(0092119) M. J. Lighthill, ``{Introduction to Fourier Analysis and Generalized
Functions}," Cambridge University Press, Cambridge, 1958.



\bibitem{MF}(0059774) P. M. Morse and H. Feshbach, ``{Methods of Theoretical
Physics}," Vol. 1, McGraw-Hill, New York, 1953.

\bibitem{Olver}(1429619) F. W. J. Olver, ``{Asymptotics and Special Functions}," A.
K. Peters, Wellesley, MA, 1997. (Reprinted, with corrections, of
original Academic Press edition, 1974).

\bibitem{OLCB} F. W. J. Olver, D. W. Lozier,  C. W. Clark and R. F. Boisvert
(eds.), ``{NIST Handbook of Mathematical Functions}," National
Institute of Standards and Technology, Gaithersburg, Maryland, to
appear.

\bibitem{Rudin}(0365062) W. Rudin, ``{Functional Analysis}," McGraw-Hill, New York,
1973.

\bibitem{Schwartz}(0207494) L. Schwartz, ``{Mathematics for the Physical Sciences},"
Addison-Wesley, Reading, MA, 1966.

\bibitem{Seaton}(1911183) M. J. Seaton, {\it Coulomb functions for attractive and
repulsive potentials and for positive and negative energies},
Comput. Phys. Comm.,  {\bf 146} (2002), 225--249.


\bibitem{VS}(2114198) O. Vall\'ee and M. Soares, ``{Airy Functions and
Applications to Physics}," Imperial College Press, London,
distributed by World Scientific, Singapore, 2004.

\end{thebibliography}
\end{document}